\documentclass[12pt, twoside, leqno]{article}

\usepackage{amsmath,amsthm}
\usepackage{amssymb}

\usepackage{enumerate}

\usepackage{graphicx}

\newtheorem{theorem}{Theorem}[section]
\newtheorem{corollary}[theorem]{Corollary}
\newtheorem{lemma}[theorem]{Lemma}
\newtheorem{proposition}[theorem]{Proposition}

\theoremstyle{definition}

\newtheorem{remark}[theorem]{Remark}

\numberwithin{equation}{section}

\newcommand{\Cee}{{\mathbb C}}

\newcommand{\En}{{\mathbb N}}

\newcommand{\fB}{{\mathcal B}}
\newcommand{\fK}{{\mathcal K}}
\newcommand{\fM}{{\mathcal M}}

\newcommand{\alp}{\alpha}
\newcommand{\del}{\delta}
\newcommand{\Del}{\Delta}
\newcommand{\eps}{\varepsilon}
\newcommand{\gam}{\gamma}
\newcommand{\Gam}{\Gamma}
\newcommand{\lam}{\lambda}
\newcommand{\vphi}{\varphi}

\newcommand{\ome}{\omega}

\newcommand{\wbar}[1]{\overline{#1}} 
\newcommand{\til}[1]{\tilde{#1}}

\newcommand{\lin}{\mathrm{lin}}

\begin{document}


\baselineskip=17pt


\title{On convoluters on $L^p$-spaces}

\author{Matthew Daws\\
Leeds, United Kingdom \\
E-mail: matt.daws@cantab.net
\and 
Nico Spronk \\
Department of Pure Mathematics \\ 
University of Waterloo\\
Waterloo, Ontario, N2L3G1, Canada\\
E-mail: nspronk@uwaterloo.ca}

\date{\today}

\maketitle


\renewcommand{\thefootnote}{}



\renewcommand{\thefootnote}{\arabic{footnote}}
\setcounter{footnote}{0}


\begin{abstract}
We prove two theorems about convolution operators on $L^p(G)$ for a locally compact group $G$.
First, if $G$ has the approximation property, then the algebra of convoluters is the algebra of pseudo-measures.
Second, the bicommutant of the algebra of pseudo-measures is the algebra of convoluters.
\end{abstract}

\section{Introduction}

Let $G$ be a locally compact group and $L^p(G)$ denote the $L^p$-space with respect to left Haar measure, for
$p\in [1,\infty]$.  The algebra of bounded operators $\fB(L^p(G))$ for $p\in(1,\infty)$ admits a natural predual
(which we specify in Section \ref{sec:predual}, below) and hence admits a specified weak* topology
which is finer than the weak operator topology.
We consider the two subalgebras,  the {\it pseudo-measures} and {\it convoluters}, which are given by
\begin{align*}
PM_p(G)&=\wbar{\lin}^{w*}\lam_p(G)&&\text{ (weak* closed linear span)} \\
CV_p(G)&=\rho_p(G)'&&\text{ (commutant)}
\end{align*}
where the right and left regular representations $\lam_p,\rho_p:G\to\fB(L^p(G))$ are given by
\[
\lam_p(s)f(t)=f(s^{-1}t),\quad
\rho_p(s)f(t)=\Del(s)^{1/p}f(ts)
\]
for all $s$ and a.e.\ $t$.  Here $\Del$ is the Haar modular function.  Since $\lam_p(G)\subseteq\rho_p(G)'$
and commutant algebras are weak operator closed, hence
weak* closed, we have that $PM_p(G)\subseteq CV_p(G)$.  Notice, moreover,
that the invertible isometry $U$ in $\fB(L^p(G))$, given by $Uf(t)=\Del^{1/p}(t^{-1})f(t^{-1})$, is self-inverse and
intertwines $\lam_p$ and $\rho_p$, which shows we may interchange the roles of left and right.

The {\it approximation property} was defined by Haagerup and Kraus \cite{haagerupk}.  A detailed definition is 
provided in Section \ref{sec:proof1}, below.

\begin{theorem}\label{theo:approx}
If $G$ has the approximation property, then 
\[
CV_p(G)=PM_p(G).
\]
\end{theorem}

A complete description of connected groups with the approximation property is provided
by Haagerup, Knudby and de Laat \cite{haagerupkdl}: any simple Lie quotient must have
real rank $1$.   This builds on an enormous body of work, including \cite{lafforguedls,haagerupdl1,haagerupdl2}.
The approximation property
passes to closed subgroups, extension groups, and is passed up from lattices \cite{haagerupk}.
In particular, countably generated free groups enjoy this property, even the stronger property of
{\it weak amenability}, defined by de Canni\`{e}re and Haggerup \cite{decanniereh}.
The approximation property does not pass to general quotients:  
the finitely generated lattice $\mathrm{SL}_3(\mathbb{Z})$ in $\mathrm{SL}_3(\mathbb{R})$
is a quotient of some free group $\mathbb{F}_n$.  In fact, we have the following implications
of properties of $G$:
\[
\text{amenable }\Rightarrow\text{ weakly amenable }\Rightarrow\text{ approximation property.}
\]
The first implication is given by Leptin \cite{leptin}, i.e.\ a bounded approximate identity in 
the Fourier algebra $A_2(G)$ satisfies the definition of weak amenability of \cite{decanniereh}. 
For amenable $G$, Theorem \ref{theo:approx}, is proved by Herz \cite{herz3}; see 
Remark \ref{rem:weakamen}.

A certain $p$-approximation property, which is implied by the approximation property
was defined by An, Lee and Ruan \cite{anlr}, and studied vigorously by Vergara \cite{vergara}.
Using this, Vergara strengthens our Theorem 
\ref{theo:approx}, but relies heavily on our methods. See Remark \ref{rem:vergara}, below.

On a related note, we also give an elementary proof of the following.  Recall that the intertwiner
$U$ is given above

\begin{theorem}\label{theo:doublecommutant}
We have commutation results:  
\[
PM_p(G)'=U\,CV_p(G)U\text{ and }(U\,CV_p(G)U)'=CV_p(G).
\]
In particular, we have bicommutant $PM_p(G)''=CV_p(G)$.
\end{theorem}

For $p=2$ the comutation result of convoluters was proved by Dixmier \cite{dixmier} using left 
Hilbert algebras.  Recently, Pham \cite{pham}
gave an elementary proof in the $p=2$ case.  
This proof is similar to the one offered here, but we wish to note that ours was conducted independently and 
posted on {\tt arXiv} in 2016.  

The present article is an updated version of our work \cite{dawss}, the main body of which was first
posted to the {\tt arXiv} in 2013.  Both theorems appear to be folklore --- see Cowling \cite{cowling}, and
Herz \cite{herz3} --- but we have been unable to track down complete proofs.
Theorem \ref{theo:approx} was used in recent work of \"{O}ztop and the 
second named author \cite{oztops} to obtain, for a pair of groups $G$ and $H$ with approximation property, 
a tensor product description of $A_p(G\times H)$.
Given recent developments related to this work and interest in it, in particular \cite{pham,vergara}, 
we have elected to update our results and submit them for publication.

\section{On preduals of pseudo-measures and convolvers}\label{sec:predual}

The entire goal of the present section is to present background for the proofs of our two
main theorems.  We also redevelop some methods of Cowling \cite{cowling}, who gave an
innovative and insightful description of a predual of $CV_p(G)$.  The present methods
give a new perspective and help to illuminate the role of the Herz-Schur multipliers which
we discuss in \ref{ssec:HS}, below.

We let $p'$ be the conjugate index to $p$, so $\frac{1}{p}+\frac{1}{p'}=1$.  We recall that $\fB(L^p(G))$ is the
dual of the space $N^p(G)=L^{p'}(G)\hat{\otimes}L^p(G)$ by way of dual pairing
$\langle T,\xi\otimes f\rangle=\langle \xi,Tf\rangle=\int_G \xi(t)[Tf](t)\,dt$.  Elements 
$\ome=\sum_{n=1}^\infty \xi_n\otimes f_n$ of $N^p(G)$ may be viewed as
functions $\ome(s,t)=\sum_{n=1}^\infty \xi_n(s) f_n(t)$ 
up to marginal almost everywhere (m.a.e.)\ equivalence; i.e.\ excepting marginally null sets,
$(N_1\times G)\cup(G\times N_2)$, where $N_1$ and $N_2$ are Haar null sets.

We let $P:N^p(G)\to A_p(G)\subseteq C_0(G)$ be given for $s$ in $G$ by 
\[
P\left(\sum_{n=1}^\infty \xi_n\otimes f_n\right)(s)
=\sum_{n=1}^\infty\langle \xi_n,\lam_p(s)f_n\rangle=\sum_{n=1}^\infty \xi_n\ast\check{f}_n
\] 
where $\check{f}(t)=f(t^{-1})$ a.e., i.e.\ $P\ome(s)=\int_G \ome(t,s^{-1}t)\,dt$.  Hence $A_p(G)$
is normed as the coimage of $P$, i.e.\ $A_p(G)\cong N^p(G)/\ker P$ isometrically.
It is straightforward to see that
$PM_p(G)$ and $CV_p(G)$ admit the following pre-annihilators in $N^p(G)$:
\begin{align*}
^\perp PM_p(G)&=\ker P \\
^\perp CV_p(G)&=\wbar{\lin}\{\rho_{p'}(t^{-1})\xi\otimes f-\xi\otimes \rho_p(t)f: \\
&\phantom{mmmmmmm}t\in G,\xi\in L^{p'}(G),f\in L^p(G)\}
\end{align*}
where we note that adjoint of these right translation operators is given by $\rho_p(t)^*=\rho_{p'}(t^{-1})$.

It is standard that
\[
PM_p(G)\cong A_p(G)^*\text{ and }CV_p(G)\cong [N^p(G)/{^\perp CV_p(G)}]^*
\]
which gives us distinguished preduals of these algebras of operators.

\subsection{$N^p(G)$ as an algebra}
Herz \cite{herz1} showed that $A_p(G)$ is always a subalgebra of $C_0(G)$, and further showed
(\cite{herz2}) that it has Gelfand spectrum given by evaluations at points of $G$.  We note that
$A_p(G)$ is known as the {\it Fig\`{a}-Talamanca--Herz algebra}.
Herz's technique for showing that $A_p(G)$ is an algebra lifts naturally to $N^p(G)$.
We note that the amplification
$T\mapsto T\otimes I:\fB(L^p(G))\to \fB(L^p(G;L^p(G)))\cong \fB(L^p(G\times G))$  is an isometry.
We consider the {\it fundamental isometry} $W_p$ in $L^p(G\times G)$
given for a.e.\ $(s,t)$ by
\[
W_pf(s,t)=f(s,st).
\]
We note that $W_p$ is invertible with $W_p^*=W_{p'}^{-1}$.  We define a {\it co-product} on
$\fB(L^p(G))$, $\Gam:\fB(L^p(G))\to\fB(L^p(G\times G))$ by
\[
\Gam(T)=W_p^{-1}(T\otimes I)W_p.
\]
This spatially implemented map is evidently weak$^*$-weak$^*$ continuous, and hence admits a preadjoint. 
We note that $N^p(G)\otimes N^p(G)$ comprises a dense subspace of $N^p(G\times G)$
given on elementary tensors of elementary tensors (``really elementary tensors") by
$(\xi\otimes f)\otimes(\eta\otimes g)\mapsto (\xi\otimes \eta)\otimes (f\otimes g)$.
Let us compute $\Gam_*$ on these really elementary tensors.  We have
\begin{align*}
\langle T,\Gam_*((\xi\otimes f)\otimes&(\eta\otimes g))\rangle
=\langle W_p^{-1}(T\otimes I)W_p,(\xi\otimes f)\otimes(\eta\otimes g)\rangle \\
&=\langle W_{p'}(\xi\otimes\eta),(T\otimes I)W_p(f\otimes g)\rangle \\
&=\int_G\int_G \xi(s)\eta(st)(T\otimes I)[(s,t)\mapsto f(s)g(st)]\,ds\,dt \\
&=\int_G \langle \xi\eta_t,T(fg_t)\rangle \,dt
\end{align*}
where $\eta_t(s)=\eta(st)$, for example.  Hence we see that
\begin{equation}\label{eq:product}
\Gam_*((\xi\otimes f)\otimes(\eta\otimes g))=\int_G  (\xi\eta_t)\otimes(fg_t) \,dt.
\end{equation}
The equation
\[
\int_G\int_G  (\xi\eta_t\vartheta_s)\otimes(fg_th_s) \,dt\,ds=
\int_G\int_G  (\xi\eta_t\vartheta_{ts})\otimes(fg_th_{ts}) \,ds\,dt
\]
shows that  $\Gam_*$ is an associative product.  We note that though $\Gam(\lam_p(r))=\lam_p(r)\otimes\lam_p(r)$,
which implies that $\Gam|_{PM_p(G)}$ is cocommutative, we have for a multiplication operator
$M_\vphi$, where $\vphi\in L^\infty(G)$, that $\Gam(M_\vphi)=M_\vphi\otimes I$, which shows that 
$\Gam$ is not generally cocommutative.  Hence $\Gam_*$ is not commutative.

Let $\fK(G)$ denote the family of all compact neighbourhoods of the identity in $G$.  For $K$ in $\fK(G)$
we let $L^p(K)=1_KL^p(G)$, which is a 1-complemented subspace.  Likewise we define
$L^{p'}(K)$ and 
\[
N^p(K)=L^{p'}(K)\hat{\otimes}L^p(K)
\]
which is a subspace of $N^p(G)$.

\begin{proposition}\label{prop:ideals}
The space $^\perp PM_p(G)=\ker P$ is a $\Gam_*$-ideal in $N^p(G)$, while each of
$^\perp CV_p(G)$ and $N^p(K)$, for $K$ in $\fK(G)$, are right ideals.
\end{proposition}

\begin{proof}
That $P:N^p(G)\to C_0(G)$ is $\Gam_*$-pointwise multiplicative is shown in \cite{herz1}, hence
$^\perp PM_p(G)=\ker P$ is an ideal.  Let us consider the case of $^\perp CV_p(G)$, manually.
On really elementary tensors we have
\begin{align*}
\Gam_*(&(\rho_{p'}(s^{-1})\xi\otimes f-\xi\otimes \rho_p(s)f)\otimes (\eta\otimes g)) \\
&=\int_G[(\rho_{p'}(s^{-1})\xi)\eta_t\otimes fg_t-\xi\eta_t\otimes(\rho_p(s)f)g_t]\,dt \\
&=\int_G \rho_{p'}(s^{-1})(\xi\eta_{st})\otimes fg_t\;dt
-\int_G \xi\eta_t\otimes \rho_p(s)(fg_{s^{-1}t})\;dt \\
&=\int_G [\rho_{p'}(s^{-1})(\xi\eta_{st})\otimes fg_t-\xi\eta_{st}\otimes \rho_p(s)(fg_t)]\;dt
\end{align*}
which is an integral of elements in $^\perp CV_p(G)$.  That each $N^p(K)$ is a right
ideal follows readily from (\ref{eq:product}).  Indeed if $\xi\otimes f\in N^p(K)$ then
\begin{align*}
(1_K\otimes 1_K)\Gam_*((\xi\otimes f)\otimes(\eta\otimes g))&=\int_G1_K\xi\eta_t\otimes 1_Kfg_t\,dt \\
&=\Gam_*((\xi\otimes f)\otimes(\eta\otimes g))
\end{align*}
so $\Gam_*((\xi\otimes f)\otimes(\eta\otimes g))\in N^p(K)$.
\end{proof}

It will be convenient below, to consider the following space of elements.  We let
$L^1_{\mathrm{loc}}(G)$ denote the space of locally a.e.\ equivalence classes of measurable functions on
$G$ which are integrable on any compact set.  An application of H\"{o}lder's inequality shows that
$L^q(G)\subseteq L^1_{\mathrm{loc}}(G)$ for any $q$ in $[1,\infty]$.  If $h\in L^1_{\mathrm{loc}}(G)$
and $g\in C_c(G)$ then $h\ast g$ is well-defined as an element of $L^1_{\text{loc}}(G)$.
We let
\[
CV_p^1(G)=\bigl\{h\in L^1_{\mathrm{loc}}(G):\sup\{\|h\ast g\|_p:g\in C_c(G),\;\|g\|_p\leq 1\}<\infty\bigr\}.
\]
Each element $h$ of $CV_p^1(G)$ defines an operator $\lam_p(h)$ in $\fB(L^p(G))$.
By testing on elements $\rho_{p'}(s^{-1})g\otimes f-g\otimes \rho_p(s)f$ where $f,g\in C_c(G)$,
and applying a standard density argument,
we see that $\lam_p(h)\in CV_p(G)$.

The next result is really a repackaging of the main result of Cowling \cite{cowling}, and is used
extensively through the rest of this note.  The second part is a localization theorem.  

\begin{theorem}\label{theo:cowling}
\begin{enumerate}[\upshape (i)]
\item  Each $T$ in $CV_p(G)$ may be approximated in the strong operator topology by a bounded
net of elements  $\lam_p(h_i)$ where each $h_i\in CV_p^p(G)=CV_p^1\cap L^p(G)$.

\item  For each $K$ in $\fK(G)$, we have that 
\[
^\perp PM_p(G)\cap N^p(K)={^\perp CV_p(G)}\cap N^p(K).
\]
\end{enumerate}
\end{theorem}

\begin{proof}
(i)  Let $T\in CV_p(G)$.  Let $(f_i)\subset C_c(G)$ be a contractive approximate identity for
$L^1(G)$.  Then for $g$ in $C_c(G)$ we have 
\[
T(f_i\ast g)=T(\rho_p(\Del^{-1/p}\check{g})f_i)=
\rho_p(\Del^{-1/p}\check{g})T(f_i)=T(f_i)\ast g
\]
where $\check{g}(t)=g(t^{-1})$ a.e.  Notice then, that
\begin{align*}
&\|T(f_i)\ast g\|_p=\|T(f_i\ast g)\|_p\leq \|T\|\|f_i\ast g\|_p\leq \|T\|\|g\|_p \\
&\text{and }\|T(f_i)\ast g -Tg\|_p=\|T(f_i\ast g-g)\|_p\overset{i}{\longrightarrow}0.
\end{align*}
It follows that $h_i=T(f_i)$ gives the desired net of elements in $CV^p_p(G)$.

(ii)  Since $^\perp PM_p(G)\supseteq {^\perp CV_p(G)}$, it suffices to show that 
$^\perp PM_p(G)\cap N^p(K)\subseteq {^\perp CV_p(G)}$.  From (i), above,
it suffices to show that for any $h$ in $CV_p^1(G)$, that $\lam_p(h)$ annihilates 
$^\perp PM_p(G)\cap N^p(K)$.

Now let $\xi\otimes f\in N^p(K)$, and let $(f_n)$ be a sequence in $C_c(G)$ whose supports are in the
neighbourhood $KK^{-1}$ of $K$ and converges to $f$ in $L^p(G)$.  Then we have that
\begin{align*}
\langle \xi, \lam_p(h) f\rangle &=\lim_n \langle \xi, \lam_p(h) f_n\rangle=\lim_n\int_K\xi(t)h\ast f_n(t) \,dt \\
&=\lim_n \int_K\xi(t)\int_{K^2K^{-1}}h(s)f_n(s^{-1}t)\, ds\,dt \\
&=\lim_n\int_{K^2K^{-1}}h(s)\int_K \xi(t)f_n(s^{-1}t)\, dt\,ds \\
&=\lim_n\int_{K^2K^{-1}}h(s)P(\xi\otimes f_n)(s)\,ds \\
&=\int_{K^2K^{-1}} h(s)P(\xi\otimes f)(s)\, ds
\end{align*}
where each interchange of integrals is justified by Fubini's theorem, as $(\xi\otimes h)1_{K\times K^2K^{-1}}$
is integrable and each $f_n$ is bounded, and the limits are justified by $L^p$-convergence of $f_n$, then
uniform convergence of $P(\xi\otimes f_n)$.  Hence, if $\ome=\sum_{n=1}^\infty \xi_n\otimes f_n\in
{^\perp PM_p(G)}\cap N^p(K)$, then uniform convergence of $\sum_{n=1}^\infty P(\xi_n\otimes f_n)=P\ome=0$
provides that
\[
\langle \lam_p(h),\ome\rangle=\sum_{n=1}^\infty \langle \xi_n,\lam_p(h)f_n\rangle
=\sum_{n=1}^\infty \int_{K^2K^{-1}}h(s)P(\xi_n\otimes f_n)(s)\,ds=0
\]
which is the desired annihilation condition.
\end{proof}

We note that proving Theorem  \ref{theo:approx}  amounts, in effect, to showing that
$\lam_p(CV_p^1(G))\subseteq PM_p(G)$, even that $\lam_p(CV_p^p(G))\subseteq PM_p(G)$.
Attacking this directly seems very delicate; see Cowling's approach \cite{cowling0} for certain simple
connected Lie groups with real rank $1$.
We will proceed through a different sequence of observations.

\begin{corollary}\label{cor:cowling1}
\begin{enumerate}[\upshape (i)]
\item   The space $\bigcup_{K\in\fK(G)}{^\perp PM_p(G)}\cap N^p(K)$ is dense in $^\perp CV_p(G)$.

\item Every $\Gam_*$-commutator is an element of $^\perp CV_p(G)$.  Hence $^\perp CV_p(G)$
is a two-sided $\Gam_*$-ideal.
\end{enumerate}
\end{corollary}

\begin{proof}
(i) We observe that the family of elements
\[
\bigcup_{K\in\fK(G)}\{\rho_{p'}(s^{-1})\xi\otimes f-\xi\otimes \rho_p(s)f:s\in K,\xi\in L^{p'}(K),f\in L^p(K)\}
\]
is contained in $\bigcup_{K\in\fK(G)}{^\perp CV_p(G)}\cap N^p(KK^{-1}\cup K^2)$, and its linear span is dense in 
$^\perp CV_p(G)$.

(ii)  Any commutator
$\Gam_*(\ome\otimes\mu-\mu\otimes\ome)$ can be approximated by elements of the form 
$\Gam_*([1_K\otimes 1_K]\ome\otimes[1_K\otimes 1_K]\mu-[1_K\otimes 1_K]\mu\otimes[1_K\otimes 1_K]\ome)$
where $K\in\fK(G)$. Such elements are in $^\perp PM_p(G)\cap N^p(K)$, hence in $^\perp CV_p(G)$.

Now, if $\ome\in{^\perp CV_p(G)}$ and $\mu\in N^p(G)$ we have
\[
\Gam_*(\mu\otimes \ome)=\Gam_*(\ome\otimes\mu)+\Gam_*(\mu\otimes \ome-\ome\otimes\mu)\in{^\perp CV_p(G)}
\]
which establishes the second fact.
\end{proof}

Hence $\Gam_*$ induces a commutative product on 
\[
\wbar{A_p}(G)=N^p(G)/{^\perp CV_p(G)}
\]
which is a predual of $CV_p(G)$.  We let $\wbar{P}:N^p(G)\to\wbar{A_p}(G)$ denote the quotient map.
We note that there is a further natural quotient map $\wbar{P}(\ome)\mapsto P\ome:\wbar{A_p}(G)\to A_p(G)$.
It is shown by Cowling \cite{cowling} that this is the Gelfand transform.  In particular,
semisimplicty of $\wbar{A_p}(G)$ is equivalent to having $^\perp CV_p(G)={^\perp PM_p(G)}$,
i.e.\ $CV_p(G)=PM_p(G)$.

Now we let for $K$ in $\fK(G)$
\[
A_p(K)=P(N^p(K)).
\]
We remark that if $K$ is an open subgroup of $G$, then our definition of $A_p(K)$ coincides with the usual definition.
For $a$ in $A_p(K)$
\[
\|a\|_{A_p(K)}=\inf\{\|\ome\|_{N^p(G)}:\ome\in N^p(K),\;a=P(\ome)\}
\]
is a norm on $A_p(K)$.

The set $\bigcup_{K\in\fK(G)}A_p(K)$ is the algebra $A_{p,c}(G)$ of compactly supported elements of $A_p(G)$.
Indeed, if $u$ in $A_{p,c}(G)$ is supported on compact $S$, let $L\in\fK(G)$ and 
$u_L=\frac{1}{m(L)}1_{SL}\ast\check{1}_L$
is $1$ on $S$, hence $u=uu_L\in A_p(K)$ whenever $SL\cup L\subseteq K$.

The following will play a critical role in the proof of Lemma \ref{lem:haagerupk}, and hence in the proof
of Theorem \ref{theo:approx}.

\begin{corollary}\label{cor:cowling2}
Let $N^p_c(G)=\bigcup_{K\in\fK}N^p(K)$.  Then $\wbar{P}(N^p_c(G))$ is a dense ideal in $\wbar{A_p}(G)$
which is algebraically isomorphic to $A_{p,c}(G)$.  Furthermore, for $\ome$ in $N^p_c(G)$ we have
\[
\|\wbar{P}(\ome)\|_{\wbar{A_p}(G)}=\inf\{\|P(\ome)\|_{A_p(K)}:\ome\in N^p(K)\}
\]
and the pairing with $CV_p(G)$ depends only on $P(\ome)$; i.e.\ if $T\in CV_p(G)$ we may write
\begin{equation}\label{eq:dpair}
\langle T,\ome\rangle=\langle T,P\ome\rangle.
\end{equation}
\end{corollary}

\begin{proof}
It is evident that $N^p_c(G)$ is a dense right $\Gam_*$-ideal in $N^p(G)$, so
$\wbar{P}(N^p_c(G))$ is a dense ideal in $\wbar{A_p}(G)$, by part (ii) of the prior corollary.
We use part (i) of the prior corollary to see that
\begin{align*}
\|\wbar{P}(\ome)\|_{\wbar{A_p}(G)}&=\mathrm{dist}(\ome,{^\perp CV_p(G)}) \\
&=\inf_{K\in\fK(G)}\mathrm{dist}(\ome,{^\perp PM_p(G)}\cap N^p(K)) \\
&=\inf\{\|P(\ome)\|_{A_p(K)}:K\in\fK(G),\;\ome\in N^p(K)\}
\end{align*}
where we note that  the norm on $A_p(K)$ is given as the coimage:
$A_p(K)\cong N^p(K)/\ker P|_{N^p(K)}$, where $\ker P|_{N^p(K)}=
{^\perp PM_p(G)}\cap N^p(K)$.  Hence if, further, $P\ome=0$ then for $T$ in $CV_p(G)$
we have
\[
|\langle T,\ome\rangle|\leq \|T\|\|\bar{P}(\ome)\|_{\wbar{A_p}(G)}=0
\]
which establishes (\ref{eq:dpair}).
\end{proof}

We record the following  localization principle.
In terminology of Herz \cite{herz3}, (b), below, is the condition of being ``formally of compact support".

\begin{theorem}\label{theo:localization}
The following are equivalent:
\begin{enumerate}[\upshape (a)]
\item $CV_p(G)=PM_p(G)$;

\item for each  $u$ in $A_{p,c}(G)$ we have
\[
\|u\|_{A_p(G)}=\inf\{\|u\|_{A_p(K)}:K\in\fK(G),\;u\in A_p(K)\};\text{ and}
\]

\item there is $C>0$ such that for each  $u$ in $A_{p,c}(G)$ we have
\[
\inf\{\|u\|_{A_p(K)}:K\in\fK(G),\;u\in A_p(K)\}\leq C\|u\|_{A_p(G)}.
\]
\end{enumerate}
\end{theorem}

\begin{proof} That (a) implies (b) is immediate form the last corollary, whilst that (b) implies (c)
is obvious for $C\geq 1$.  If (c) holds, then on the dense subspace 
$A_{p,c}(G)$, which we may identify with $\wbar{P}(N^p_c(G))$, we
have that $\|\cdot\|_{A_p(G)}$ and $\|\cdot\|_{\wbar{A_p}(G)}$ are equivalent norms, so the map
$\ome+{^\perp CV_p(G)}\mapsto P(\ome)$ is injective, which implies that $^\perp CV_p(G)={^\perp PM_p(G)}$.
The latter fact gives (a).
\end{proof}

\subsection{The action of Herz-Schur multipliers}\label{ssec:HS}
Let [SQ$_p$] denote the class of spaces which are isometrically 
isomorphic to a subspace of a quotient of an $L^p$-space.
Kwapie\'{n} \cite{kwapien} noted that for $E$ in [SQ$_p$]  that the ampliation
$T\mapsto T\otimes I_E:\fB(L^p(G))\to
\fB(L^p(G;E))$ is an isometry.  In fact, he showed that this property characterizes elements of the class [SQ$_p$].
See Dales et al \cite{daleslot} for an exposition on this.
Hence, by duality, the map $\gam_E:L^{p'}(G;E^*)\hat{\otimes}L^p(G;E)\to N^p(G)$
given on elementary tensors m.a.e.\ by
\begin{equation}\label{eq:gammaE}
\gam_E(\Xi\otimes F)(s,t)=\langle \Xi(s),F(t)\rangle_{E^*,E}
\end{equation}
is a quotient map, in particular a contraction.   Also see Herz \cite{herz2}.
Let us define the {\it $p$-Herz-Schur multiplier} space by
\begin{align*}
\fM_p(G)=\{\vphi:G\to\Cee\;|&\;\vphi(st^{-1})=\langle\beta(s),\alp(t)\rangle\text{ where each} \\
&\alp:G\to E,\beta:G\to E^*\text{ is continuous}\\
&\text{and bounded for some }E\text{ in [SQ}_p]\}
\end{align*}
These were essentially described by Herz \cite{herz3,herz4}, and the specific version we are using here
was given by the first named author \cite{daws}.
As shown in \cite{daws}, $\fM_p(G)$ is the algebra of ``$p$-completely
bounded multipliers" of $A_p(G)$, and a Banach algebra with respect to the norm
\[
\|\vphi\|_{\fM_p}=\inf\{\|\beta\|_\infty\|\alp\|_\infty:
\vphi(st^{-1})=\langle \beta(s),\alp(t)\rangle\text{ as above}\}.
\]

\begin{proposition} \label{prop:HSaction}
\begin{enumerate}[\upshape (i)]
\item $N^p(G)$ is a Banach $\fM_p(G)$-module via the action given m.a.e.\ by
\[
\vphi\cdot \ome(s,t)=\vphi(st^{-1})\ome(s,t).
\]
\item  Each of $^\perp PM_p(G)$ and $N^p(K)$ ($K\in\fK(G)$) are $\fM_p(G)$-submodules,
hence so too is $^\perp CV_p(G)$.
\end{enumerate}
\end{proposition}

\begin{proof}
(i) This is an immediate application of the map $\gam_E$ in (\ref{eq:gammaE}), above.  
If $\vphi\in\fM_p(G)$ with $\vphi(st^{-1})=\langle \beta(s),\alp(t)\rangle$, and
$\xi\otimes f$ is an elementary tensor in $N^p(G)$, then let $\xi\beta(s)=\xi(s)\beta(s)$ in $L^{p'}(G;E^*)$ and
$f\alp(t)=f(t)\alp(t)$ in $L^p(G;E)$.  Notice that $\|\xi\beta\|_{p';E^*}\leq \|\xi\|_{p'}\|\beta\|_\infty$ and
$\|f\alp\|_{p;E}\leq \|f\|_p\|\alp\|_\infty$.  Then
\[
\vphi\cdot(\xi\otimes f)=\gam_E(\xi\beta\otimes f\alp)\in N^p(G)
\]
and hence
\[
\|\vphi\cdot(\xi\otimes f)\|_{N^p}\leq 
\|\xi\beta\|_{p';E^*}\|f\alp\|_{p;E}\leq \|\beta\|_\infty\|\alp\|_\infty\|\xi\|_{p'}\|f\|_p
\]
from which it follows that $\|\vphi\cdot(\xi\otimes f)\|_{N^p}\leq \|\vphi\|_{\fM_p}\|\xi\otimes f\|_{N^p}$.

(ii)  If $\vphi\in\fM_p(G)$ and $\ome\in\ker P={^\perp PM_p(G)}$ then
\[
P(\vphi\cdot\ome)(s)=\int_G\vphi(t(s^{-1}t)^{-1})\ome(t,s^{-1}t)\,dt=\vphi(s)P\ome(s)=0.
\]
It is facile that each $N^p(K)$ is an $\fM_p(G)$-submodule.  Thus each
$^\perp PM_p(G)\cap N^p(K)$ is an $\fM_p(G)$-submodule, and the result for
$^\perp CV_p(G)$ follows from the density result, Corollary \ref{cor:cowling1} (i), above.
\end{proof}

It is immediate that the action in (ii), above, induces actions of $\fM_p(G)$ on each of
$A_p(G)$ and $A_p(K)$ (each by pointwise multiplication) and on $\wbar{A_p}(G)$, 
making each a continuous module.  The spaces $A_p(K)$ are shown to be Banach $A_p(G)$-modules
in \cite{oztops}, by a method quite similar to that above.  We note that the natural map
$\wbar{P}(\ome)\mapsto P\ome:\wbar{A_p}(G)\to A_p(G)$ is an $\fM_p(G)$-module map.

We shall have need to consider the adjoint action on the dual spaces.  That $CV_p(G)$
is a $\fM_p(G)$-module seems also to be shown by Herz \cite{herz4}.

\begin{corollary}\label{cor:HSaction1}
The algebras $CV_p(G)$ and $PM_p(G)$ are dual Banach $\fM_p(G)$-modules:
\[
\langle \vphi\cdot T,\ome\rangle=\langle T,\vphi\cdot\ome\rangle.
\]
If $\vphi\in \fM_p(G)$ and $h\in CV_p^1(G)$, then $\vphi h\in CV_p^1(G)$ and 
\[
\vphi\cdot\lam_p(h)=\lam_p(\vphi h).
\]
\end{corollary}

\begin{proof}
The first statement being an evident adjoint module operation, we are left only to inspect the second statement.
However, if $\xi\otimes f$ is an elementary tensor in $N^p_c(G)$, i.e.\ in $N^p(K)$ for some $K$, 
then arguments just as in the proof
of Theorem \ref{theo:cowling} (ii) provide that
\begin{align*}
\langle \vphi\cdot\lam_p(h), \xi\otimes f\rangle 
&=\langle \lam_p(h),\vphi\cdot (\xi\otimes f)\rangle \\
&=\int_{KK^{-1}}  h(s) P(\vphi\cdot(\xi\otimes f))(s)\,ds \\
&=\int_{KK^{-1}} \vphi(s)h(s)\int_K \xi(t)f(s^{-1}t)\, dt\,ds \\
&=\langle \xi,(\vphi h)\ast f\rangle
\end{align*}
where the last interchange of integrals is justified as 
\[
(s,t)\mapsto |\vphi(s)h(s)\xi(t)f(s^{-1}t)| 
\]
is  integrable thanks to H\"{o}lder's inequality and Tonelli's theorem.  If $f\in C_c(G)$ then
\begin{align*}
\|(\vphi h)\ast f\|_p&=\sup\{|\langle \vphi\cdot\lam_p(h), \xi\otimes f\rangle|:\xi\in L^p(K),K\in\fK(G),\|\xi\|_p\leq 1\} \\
&\leq \|\vphi\cdot \lam_p(h)\|\|f\|_p. 
\end{align*}
Then taking supremum over such $f$ with $\|f\|_p\leq 1$, reveals that $\vphi h\in CV_p^1(G)$, and
further gives the desired formula.
\end{proof}

We note that $A_p(G)\subseteq \fM_p(G)$, is an ideal within the latter space, 
and the inclusion is a contractive map.  Indeed
if $a=\sum_{n=1}^\infty \langle \xi_n,\lam_p(\cdot)f_n\rangle$ in $A_p(G)$, with 
$\sum_{n=1}^\infty\|\xi_n\|_{p'}\|f_n\|_p<\infty$ and each $\|\xi_n\|_{p'}\|f_n\|_p>0$, let
$\xi_n'=\|\xi_n\|_{p'}^{-1/p}\|f_n\|_p^{1/p'}\xi_n$ and $f_n'=\|\xi_n\|_{p'}^{1/p}\|f_n\|_p^{-1/p'}f_n$.
Then $\beta(s)=(\lam_{p'}(s^{-1})\xi_n')_{n=1}^\infty$ in $\ell^{p'}(\En,L^{p'}(G))$ and 
$\alp(t)=(\lam_p(t^{-1})f_n')_{n=1}^\infty$ in $\ell^p(\En,L^p(G))$ show that
$a\in\fM_p(G)$, with $\|a\|_{\fM_p}\leq \|a\|_{A_p(G)}$.

Hence we get the following crucial result which will help us in the proof of Theorem \ref{theo:approx}.

\begin{corollary}\label{cor:HSaction2}
Let $T\in CV_p(G)$ and $a\in A_p(G)$.  Then $a\cdot T\in PM_p(G)$.
\end{corollary}

\begin{proof}
By a norm estimate, we may assume that $a\in A_{p,c}(G)$.  Furthermore, Theorem \ref{theo:cowling} (i)
provides that $T$ is a weak* limit of $\lam_p(h_i)$ for some $h_i$ in $CV_p^1(G)$.  
Then the prior corollary provides that each $a\cdot\lam_p(h_i)=\lam_p(ah_i)$, where $ah_i\in L^1_{\mathrm{loc}}(G)$
with compact support, hence $ah_i\in L^1(G)$, so $\lam_p(ah_i)\in PM_p(G)$.
Thus $a\cdot T=\lim_i a\cdot\lam_p(h_i)=\lim_i \lam_p(ah_i)\in PM_p(G)$.
\end{proof}

\begin{remark}\label{rem:weakamen}
We say that $G$ is {\it $p$-weakly amenable} whenever $A_p(G)$ -- if we like $A_{p,c}(G)$ -- admits
a net $(\vphi_i)$ of elements which is bounded in $\fM_p(G)$ and tends to $1$
uniformly on compact sets.  For $p=2$, this is the property of weak amenability of de Canni\`{e}re and
Haagerup \cite{decanniereh}.  Remarks in the proof of Theorem \ref{theo:approx}, in Section \ref{sec:proof1} below,
show that the assumption of $2$-weak amenability is sufficient to obtain general $p$-weak amenability.

Suppose that $G$ is $p$-weakly amenable.
Using the density of $N^p_c(G)$ in $N^p(G)$, it is easy
to conclude that $\vphi_i\cdot \ome$ tends in norm to $\ome$, for any $\ome$ in $N^p(G)$.
Hence, by the last corollary, any $T$ in $CV_p(G)$ is approximated weak* by the (bounded)
net of elements $\vphi_i\cdot T$ from $PM_p(G)$.  Hence we have proved 
Theorem \ref{theo:approx} for such groups.  For amenable groups, this is the proof of Herz \cite{herz3}.
For weakly amenable groups, this proof is hinted at by Cowling \cite{cowling}.
\end{remark}

\section{Proof of the approximation result, Theorem \ref{theo:approx}}\label{sec:proof1}

We recall that $\fM_2(G)$ is a dual space.  We identify $L^1(G)$ with its evident image
in $\fM_2(G)^*$, acting by integration, and let $Q(G)=\wbar{L^1(G)}^{\fM_2(G)^*}$, the closure
of $L^1(G)$ in $\fM_2(G)^*$.  De Canni\`{e}re and Haagerup \cite{decanniereh} show that
$\fM_2(G)\cong Q(G)^*$, and thus we define a weak*-topology on $\fM_2(G)$.

As in Haagerup and Kraus \cite{haagerupk}, we say that $G$ admits the {\it approximation property} if
$1\in\overline{A_p(G)}^{w*}$, i.e.\ there is a net $(\vphi_i)\subset A_2(G)$ 
-- we may suppose, in fact that $(\vphi_i)\subset A_{2,c}(G)$ -- 
for which $\langle 1,q\rangle=\lim_i\langle \vphi_i,q\rangle$ for $q$ in $Q(G)$.

The following claim is implicit in \cite{haagerupk} (see notes prior to Proposition 1.3, there), 
but the proof offered seems incomplete since
it is not known if translations on $\fM_2(G)$ are continuous in the translating variable.

\begin{lemma}\label{lem:convolution}
If $f\in L^1(G)$ and $\vphi\in \fM_2(G)$, then $f\ast \vphi\in \fM_2(G)$.
\end{lemma}

\begin{proof}
Since we have a contractive embedding
$\fM_2(G)\hookrightarrow L^\infty(G)$, we have that for $g$ in $L^1(G)$ that $\|g\|_Q\leq
\|g\|_1$.  Further $\fM_2(G)$ is closed under left translations, and translation operators are isometries,
so $Q(G)$ is a homogeneous space for left translation, and thus
admits left convolution by elements of $L^1(G)$.  Also,
$\fM_2(G)$ consists of weakly almost periodic functions (see the observation of Xu \cite{xu})
and hence of uniformly continuous functions, so $f\ast \vphi$ makes sense as a uniformly
continuous function.  Now we have
\begin{align*}
&\sup\left\{\left|\int_G f\ast \vphi(s)g(s)\,ds\right|:g\in L^1(G),\;\|g\|_Q\leq 1\right\} \\
&\phantom{m}=\sup\left\{\left|\int_G  \vphi(s)\tilde{f}\ast g(s)\,ds\right|:g\in L^1(G),\;\|g\|_Q\leq 1\right\} \\
&\phantom{m}\leq \|f\|_1\sup\left\{\left|\int_G  \vphi(s) g(s)\,ds\right|:g\in L^1(G),\;\|g\|_Q\leq 1\right\}
=\|f\|_1\|\vphi\|_{\fM_2}
\end{align*}
where $\tilde{f}(t)=\Del(t^{-1})f(t^{-1})$.  Hence $f\ast\vphi\in\fM_2(G)$.
\end{proof}

The fact that Hilbert spaces, which are exactly the [SQ$_2$]-spaces, are also [SQ$_p$] spaces (see \cite{kwapien})
gives a contractive embedding $\fM_2(G)\hookrightarrow\fM_p(G)$.

The following construction is modelled after Proposition 1.3 of \cite{haagerupk}.

\begin{lemma}\label{lem:haagerupk}
Fix $a$ in $A_{p,c}(G)$ with $a\geq 0$ and $\int_G a=1$.  For $T$ in $CV_p(G)$ and $\ome$ in $N^p(G)$ let
\[
q_{T,\ome}\in\fM_2(G)^*\text{ be given by }q_{T,\ome}(\vphi)=\langle T,(a\ast \vphi)\cdot\ome\rangle.
\]
Then $q_{T,\ome}\in Q(G)$.
\end{lemma}

\begin{proof}
It is evident that $\|q_{T,\ome}\|_{\fM_2^*}\leq \|T\|\|\ome\|_{N^p(G)}$.  Hence by norm approximation, 
it suffices to show that if $\ome\in N^p_c(G)$, say $\ome\in N^p(K)$ for some $K$ in $\fK(G)$, then
$q_{T,\ome}(\vphi)=\int_G \vphi g$ for some $g$ in $L^1(G)$.  
We may further assume $K$ is so large that $a\in A_p(K)$.
With these assumptions, let
$S=\mathrm{supp}(a)^{-1}\mathrm{supp}(P\ome)$, and we have
\begin{align*}
(a\ast\vphi)(s)P\ome(s)&=\int_G a(t)\vphi(t^{-1}s)P\ome(s)\,dt \\
&=\int_G a(t)1_S(t^{-1}s)\vphi(t^{-1}s)P\ome(s)\,dt \\
&=\int_G 1_S(t^{-1})\vphi(t^{-1})a_t(s)P\ome(s)\,dt.
\end{align*}
Let $L=KK^{-1}S^{-1}\cup KK^{-1}$ in $\fK(G)$, which contains the support of each
$a_tP\ome$.
We note that $t\mapsto a_t:S^{-1}\to A_p(L)$ is continuous.  Indeed, on elementary tensors
$\xi\otimes f$ in $N^p(K)$, we have that
$(\xi\ast \check{f})_t=\xi\ast(\lam_p(t)f)^\vee$ and left translation is continuous on elements of $L^p(G)$. 
Hence we realize
\[
(a\ast \vphi)P\ome=\int_G1_S(t^{-1})\vphi(t^{-1})[a_t\,P\ome]\;dt
\]
as a Bochner integral in $A_p(L)$, i.e.\ respecting the norm $\|\cdot\|_{A_p(L)}$.  Then let
\[
g(t)=1_S(t)\Del(t^{-1})\langle T,a_{t^{-1}}\, P\ome\rangle
\]
which defines an element of $L^1(G)$.  We may thus use the dual pairing (\ref{eq:dpair}) to compute
\begin{align*}
\int_G \vphi g&=\int_G 1_S(t^{-1})\vphi(t^{-1})\langle T,a_t\, P\ome\rangle\,dt \\
&=\left\langle T,\int_G 1_S(t^{-1})\vphi(t^{-1})[a_t\,P\ome]\right\rangle=\langle T, (a\ast \vphi)\cdot\ome\rangle.
\end{align*}
This establishes the desired claim.
\end{proof}

\begin{proof}[Proof of Theorem \ref{theo:approx}]
Let $a$ be as in Lemma \ref{lem:haagerupk} and $(\vphi_i)\subset A_{2,c}(G)$ be a net converging
weak* to 1.  Note that $A_{2,c}(G)\subseteq \fM_2(G)\subseteq \fM_p(G)$, and hence $A_{2,c}(G)$
is a space of compactly supported elements of $\fM_p(G)$.  The argument just before Corollary \ref{cor:cowling2}
tells us that any compactly supported element of $\fM_p(G)$ is in $A_{p,c}(G)$.  Since $A_p(G)$ 
is a homogeneous space for left translations, we find that each $a\ast \vphi_i\in A_p(G)$.

Given $T$ in $CV_p(G)$, let $q_{T,\ome}$ be as in Lemma \ref{lem:haagerupk}.  Then we see that
\begin{align*}
\langle (a\ast\vphi_i)\cdot T,\ome\rangle
&=\langle T,(a\ast\vphi_i)\cdot\ome\rangle=q_{T,\ome}(\vphi_i) \\
&\overset{i}{\longrightarrow} q_{T,\ome}(1)
=\langle T,(a\ast 1)\cdot\ome\rangle
=\langle T,\ome\rangle
\end{align*}
since $a\ast 1=\left[\int_G a\right]1=1$, by assumption on $a$.  But Corollary \ref{cor:HSaction2}
provides that each $(a\ast\vphi_i)\cdot T\in PM_p(T)$, and hence so too the weak* limit of this net, $T$,
is also in $PM_p(G)$.
\end{proof}

Hints of the Theorem \ref{theo:approx} are given by Cowling \cite{cowling}.  Given how heavily we
have exploited his results, i.e.\ Theorem \ref{theo:cowling}, we suspect he knew how to prove this.

\begin{remark}\label{rem:vergara}
The space $\fM_p(G)$ generally has a predual $Q_p(G)$, as was verified by T. Miao in an unpublished
manuscript.  Using this, An, Lee and Ruan \cite{anlr} defined the {\it $p$-approximation property} ($p$-AP), which is
having $1$ in the weak* closure of $A_p(G)$ in $\fM_p(G)$.  As verified by Vergara \cite{vergara},
$p$-AP is equivalent to $p'$-AP, and if $2\leq p\leq q<\infty$, then $p$-AP implies $q$-AP.  Hence
$p$-AP for $p\not=2$ is ostensibly more general than approximation property ($2$-AP).
However, it is shown in \cite{vergara} that certain canonical higher rank simple Lie groups fail $p$-AP
for any $p$ -- the same ones known to  Lafforgue and de la Salle \cite{lafforguedls} and
Haagerup and de Laat \cite{haagerupdl1} --
giving convincing evidence that, at least for connected groups and lattices within,
that $p$-AP is equivalent to $2$-AP.  No examples are known of groups admitting $p$-AP
for some $p$, but not $2$-AP.

Nonetheless, it is simple to modify the Lemmas of this section to accomodate the
assumption of $p$-AP.  We refer the reader \cite{vergara} for details.
\end{remark}

\section{Proof of the commutation results,  Theorem \ref{theo:doublecommutant}}\label{sec:bicommutant}

Theorem \ref{theo:cowling}, above, provides a very useful approximation of a convoluter $T$ in $CV_p(G)$
by a bounded net of operators $(\lam_p(h_i))$ where each $h_i\in CV_p^p(G)$.  If we were
conducting approximations for $p=2$, then we need only consider self adjoint $T=T^*$ and we need not worry
about determining the structure of $\lam_p(h_i)^*$.  For $p\not=2$ we have no such luck.  We overcome
this difficulty below.  Furthermore, in the spirit of Theorem \ref{theo:cowling}, we give a careful
pointwise approximation of a convoluter, on a dense subspace, by convolutions with elements of $L^1(G)$.

We note that $f\mapsto \til{f}$, where $\til{f}(t)=\Del(t^{-1})f(t^{-1})$ for (locally) a.e.\ $s$, defines an isometric involution
on $L^1(G)$, and clearly defines a linear map on $L^1_{\mathrm{loc}}(G)$ which satisfies
$(h\ast f)^\sim=\til{f}\ast\til{h}$ for any convolvable pair of elements.  
Hence if $(f_i)$ is a contractive approximate identity on $L^1(G)$, then so too is 
$(\til{f}_i)$.  Furthermore, if $f\in L^1(G)$, then $\lam_p(f)^*=\lam_{p'}(\til{f})$, as is standard to check.
We also observe that since $\rho_p(s)^*=\rho_{p'}(s^{-1})$ for all $s$ in $G$, we have
$\{T^*:T\in CV_p(G)\}=CV_{p'}(G)$.

\begin{theorem}\label{theo:convolveradj}
Let $T\in CV_p(G)$. 

\begin{enumerate}[\upshape (i)]
\item   Then there is a net of elements $(k_i)$ in $CV_p^p(G)$ for which
$(\lam_p(k_i))$ is a bounded net in $CV_p(G)$ converging strong operator to $T$, and also
$(\til{k}_i)$ is a net in $CV_{p'}^{p'}(G)$ for which $(\lam_{p'}(\til{k}_i))$ is bounded
in $CV_{p'}(G)$ and converges to $T^*$.

\item Given $g,f$ in $C_c(G)$ and $\eps>0$, there is $k$ in $C_c(G)$ for which
\[
\|Tf-\lam_p(k)f\|_p<\eps\text{ and }\|T^*g-\lam_{p'}(\til{k}) g\|_{p'}<\eps.
\]
Here $k$ is dependant upon the choices of $g,f$ and $\eps$.
\end{enumerate}
\end{theorem}

\begin{proof}
(i)  As in Theorem \ref{theo:cowling} we let $(f_i)\subset C_c(G)$ be a contractive approximate
identity for $L^1(G)$ and set $h_i=Tf_i$.  We note that
\[
\lam_p(h_i)=\lam_p(Tf_i)=T\lam_p(f_i).
\]
Let $g,f\in C_c(G)$.  Since $\til{h}_i\in L^1_{\mathrm{loc}}(G)$, $\til{h}_i\ast g\in L^1_{\mathrm{loc}}(G)$
and, letting $S=\mathrm{supp}f\cup\mathrm{supp} g$, we have
\begin{align*}
\int_G \til{h}_i\ast g(s)f(s)\, ds &=\int_S\int_{SS^{-1}} \Del(t^{-1})h_i(t^{-1})g(t^{-1}s)f(s)\, dt\, ds \\
&=\int_{SS^{-1}} \int_S g(s)h_i(t)f(t^{-1}s)\, ds\, dt \\
&=\langle g,\lam_p(h_i)f\rangle=\langle g,T\lam_p(f_i)f\rangle
=\langle \lam_{p'}(\til{f}_i)T^*g,f\rangle.
\end{align*}
Hence by density it follows that $\til{h}_i\ast g=\lam_{p'}(\til{f}_i)T^*g$, and hence $\til{h}_i\in CV_{p'}^1(G)$
with 
\[
\lam_p(\til{h}_i)=\lam_{p'}(\til{f}_i)T^*.
\]
We do not know of a means to
see that $\til{h}_i\in CV_{p'}^{p'}(G)$.  We use another convolution to overcome this difficulty.

Let $k_i=f_i\ast h_i$.  It is clear that $k_i\in CV_p^p(G)$ with $\lam_p(k_i)=\lam_p(f_i)\lam_p(h_i)$,
and hence $\|\lam_{p'}(\til{k}_i)\|\leq \|f_i\|_1\|\lam_p(h_i)\|$.  Likewise
$\til{k_i}=\til{h}_i\ast\til{f}_i\in CV_{p'}^1(G)$ with $\lam_{p'}(\til{k_i})=\lam_{p'}(\til{h}_i)\lam_p(\til{f}_i)$, with
$\|\lam_{p'}(\til{k}_i)\|\leq \|\lam_{p'}(\til{h}_i)\|\|f_1\|_1$.
Moreover, as $T^*\in CV_{p'}(G)$ we have that $T^*(\til{f}_i)\in CV_{p'}^{p'}(G)$ and 
\[
\lam_{p'}(\til{k_i})=\lam_{p'}(\til{h}_i)\lam_p(\til{f}_i)=\lam_{p'}(\til{f}_i)T^*\lam_p(\til{f}_i)
=\lam_{p'}(\til{f}_i\ast T^*\til{f}_i)
\]
so $\til{k_i}=\til{f}_i\ast T^*(\til{f}_i)\in CV_{p'}^{p'}(G)$ too.  We have that
\begin{align*}
\|\lam_p(k_i)f-Tf\|_p&=\|f_i\ast [Tf_i\ast f]-Tf\|_p \\
&\leq \|f_i\ast [Tf_i\ast f]-f_i\ast Tf\|_p+\|f_i\ast Tf-Tf\|_p \\
&\leq \|f_i\|_1 \|Tf_i\ast f-Tf\|_p+\|f_i\ast Tf-Tf\|_p
\end{align*}
which tends to zero.  By uniform boundedness of the operators $\lam_p(k_i)$ and density
of $C_c(G)$ in $L^p(G)$, we see that $\lam_p(k_i)$ tends to $T$ in the strong operator topology.
Similarly, the net $(\lam_{p'}(\til{k_i}))$ is bounded and tends to $T^*$ in the strong operator topology.

(ii)  Let $\del>0$.  Let $i$ be such that
\[
\|Tf-k_i\ast f\|_p<\del\text{ and }\|T^*g-\til{k}_i\ast g\|_{p'}<\del.
\]
Inner regularity of the Haar measure provides $K$ in $\fK(G)$ for which
\[
\|k_i-1_Kk_i\|_p<\del\text{ and }\|\til{k}_i-1_{K^{-1}}\til{k}_i\|_{p'}<\del.
\]
Let $k=1_Kk_i$ which is in $L^1\cap L^p(G)$, and note that $\til{k}=1_{K^{-1}}\til{k}_i$.  Since
$l\ast f=\rho_p(\Del^{-1/p}\check{f})l$ for $l$ in $L^p(G)$, we have that
\begin{align*}
\|Tf-k\ast f\|_p&\leq \|Tf-k_i\ast f\|_p+\|k_i\ast f-k\ast f\|_p \\
&< \del +\|\rho_p(\Del^{-1/p}\check{f})\|\del.
\end{align*}
Likewise we have that
\[
\|T^*g-\til{k}\ast g\|_{p'}< \del+\|\rho_{p'}(\Del^{-1/p'}\check{g})\|\del.
\]
Hence an obvious choice of $\del$ yields the desired inequalities.
\end{proof}

We note that the sorts of global estimates found in (i), above, were required to obtain the pointwise estimates of (ii).

\begin{remark}\label{rem:rightestimates}
As noted in the introduction, there is a self-inverting isometry $U$ on $L^p(G)$, which intertwines
$\lam_p$ and $\rho_p$:  $U\lam_p(s)=\rho_p(s)U$ for $s$ in $G$.  Hence $U^*$ intertwines
$\lam_{p'}$ and $\rho_{p'}$.  We let $CV_p'(G)=\lam_p(G)'$ so
\[
CV_p'(G)=(U\rho_p(G)U)'=U\rho_p(G)'U=U\,CV_p(G)\,U.
\]
If $S\in CV_p'(G)$ and $g,f\in C_c(G)$, then $U^*g,Uf\in C_c(G)$.  Hence, given $\eps>0$ Theorem
\ref{theo:convolveradj} (ii) provides $k'$ in $L^1(G)$ for which
\[
\|USUUf-\lam_p(k')Uf\|_p<\eps\text{ and }\|(USU)^*U^*g-\lam_{p'}(\til{k'})U^*g\|_{p'}<\eps
\]
hence multiplying the arguments of the expressions by the respective isometries $U$ and $U^*$
we obtain
\[
\|Sf-\rho_p(k')f\|_p<\eps\text{ and }\|S^*g-\rho_{p'}(\til{k'}) g\|_{p'}<\eps.
\]
\end{remark}

The key reason we have Theorem \ref{theo:convolveradj}, is that for $k,k'$ in $L^1(G)$ we
have $\lam_p(k)\rho_p(k')=\rho_p(k')\lam_p(k)$.  We do not know how to verify
directly that $\lam_p(h)\rho_p(h')=\rho_p(h')\lam_p(h)$ for $h,h'$ in $CV_p^p(G)$.

\begin{proof}[Proof of Theorem \ref{theo:doublecommutant}]
Let $PM_p'(G)=U\, PM_p(G) \, U=\wbar{\lin}^{w*}\rho_p(G)$.  We have that
\[
PM_p(G)'=\lam_p(G)'=CV_p'(G)\supseteq PM_p'(G)\text{ and  }PM'_p(G)'=CV_p(G).
\]
Thus it follows that 
\[
PM_p(G)''=CV_p'(G)'\subseteq PM'_p(G)'=CV_p(G).
\]
Hence we will be done once we show that $CV_p(G)\subseteq CV_p'(G)'$, i.e.\ if $T\in CV_p(G)$
and $S\in CV'_p(G)$, we wish to see that $TS=ST$.

To this end, fix $g,f$ in $C_c(G)$.  As in Theorem \ref{theo:convolveradj} (ii) and Remark \ref{rem:rightestimates},
given $\eps>0$, find $k,k'$ in $L^1(G)$ for which
\begin{align*}
\|Tf-\lam_p(k)f\|_p<\eps&\text{ and }\|T^*g-\lam_{p'}(\til{k}) g\|_{p'}<\eps \\
\|Sf-\rho_p(k')f\|_p<\eps&\text{ and }\|S^*g-\rho_{p'}(\til{k'}) g\|_{p'}<\eps.
\end{align*}
Then
\begin{align*}
|\langle &g,TSf\rangle-\langle g,\lam_p(k)\rho_p(k')f\rangle| \\
&\leq |\langle T^*g,Sf-\rho_p(k')f\rangle|
+|\langle T^*g-\lam_{p'}(\til{k})g,\rho_p(k')f\rangle| \\
&\leq\|T\|\|g\|_{p'}\eps+\eps\|\rho_p(k')f\|_p\leq (\|T\|\|g\|_{p'}+\|S\|\|f\|_p+\eps)\eps
\end{align*}
and, by a similar calculation, we see that
\[
|\langle g,STf\rangle-\langle g,\rho_p(k')\lam_p(k)f\rangle|
\leq (\|S\|\|g\|_{p'}+\|T\|\|g\|_p+\eps)\eps.
\]
Since $\lam_p(k)\rho_p(k')=\rho_p(k')\lam_p(k)$ for any $k,k'$ in $L^1(G)$,
we can take $\eps$ to $0$ and see that
\[
\langle g,TSf\rangle=\langle g,STf\rangle.
\]
Hence by density we see that $\langle \xi,TSf\rangle=\langle \xi,STf\rangle$ for
any $\xi$ in $L^{p'}(G)$ and $f$ in $L^p(G)$, so $TS=ST$, as desired.
\end{proof}

We obtain an ``elementary" proof of Dixmier's result \cite{dixmier}, i.e.\ 
without use of left  Hilbert algebras.  The
common algebra in the result below is usually denoted $VN(G)$ and called
the {\it group von Neuman algebra}.

\begin{corollary}\label{cor:dixmier}
We have that 
\[
CV_2(G)=CV'_2(G)'=PM_2(G)=PM'_2(G)'.
\]
\end{corollary}

\begin{proof}
Here $\lam_2(G)$ is a self-adjoint subset of $\fB(L^2(G))$, so $VN(G)=PM_2(G)$ is self-adjoint.
The bicommutant theorem of von Neumann, \cite[II.3.9]{tak1} for 
example, tells us that $PM_2(G)''$ is the $\sigma$-strong$^*$ closure of $PM_2(G)$.
As a convex subspace of $\fB(L^2(G))$ is weak$^*$-closed if and only if it is
$\sigma$-strong$^*$ closed (see \cite[II.2.6]{tak1}), we see that
$PM_2(G)'' = PM_2(G)$.  The rest follows from the theorem above.
\end{proof}


\subsection*{Acknowledgements}
The second named researcher was partially supported by an NSERC Discovery Grant.


\end{document}